\documentclass[12pt,a4paper]{article}

\usepackage{amsmath,amsthm,amssymb,amsfonts}
\usepackage{mathtools}
\usepackage{geometry}
\usepackage{hyperref}
\usepackage{enumitem}
\usepackage{booktabs}
\usepackage{array}
\usepackage{tikz}
\usetikzlibrary{calc} 
\usepackage{amsthm}
\usepackage{amssymb}
\usepackage{graphicx}
\usepackage{blindtext}
\usepackage{hyperref}
\numberwithin{equation}{section}
\newtheorem{theorem}{Theorem}[section]
\newtheorem{proposition}[theorem]{Proposition}
\newtheorem{lemma}[theorem]{Lemma}

\newtheorem{definition}[theorem]{Definition}

\newtheorem{example}[theorem]{Example}
\newtheorem{assumption}[theorem]{Assumption}

\newcommand{\WpD}{W^{1,p}_{\mathcal{V}^D}(\mathcal G)}

\newcommand{\abs}[1]{\left|#1\right|}

\title{On the \texorpdfstring{$p$}{p}-torsional rigidity of compact metric
	graphs: a sharp Kohler--Jobin inequality}
	
\author{Sedef \"Ozcan\thanks{%
		Department of Mathematics, Faculty of Science, Dokuz Eyl\"ul University,
		\.Izmir, Turkey. E-mail: \texttt{sedef.taskin@deu.edu.tr}.
		ORCID: \href{https://orcid.org/0009-0005-4636-9406}{0009-0005-4636-9406}.}}

\date{}

\begin{document}
	
	\maketitle
\begin{abstract}
	We investigate the $p$-torsional rigidity for the $p$-Laplacian,
	$1<p<\infty$, on compact connected metric graphs equipped with Dirichlet
	conditions on a nonempty set $\mathcal{V}^D$ of degree-one vertices and
nonlinear Kirchhoff conditions at all remaining vertices. We
	establish the existence, uniqueness, and positivity of the 
	$p$-torsion function, together with a variational characterization of the
	$p$-torsional rigidity. Our main contribution is the derivation of two
	sharp isoperimetric inequalities. We first prove a $p$-Saint-Venant
	inequality, showing that, among all compact metric graphs of prescribed
	total length, the $p$-torsional rigidity is maximized precisely by the
	interval with a single Dirichlet endpoint. We then derive a sharp
	$p$-Kohler--Jobin inequality, providing a scale-invariant lower bound for
	the first eigenvalue of the $p$-Laplacian in terms of the $p$-torsional
	rigidity. These results yield nonlinear counterparts, in the setting of
	compact metric graphs, of the classical Saint-Venant and Kohler--Jobin
	inequalities, and extend the linear theory, where $p=2$, developed by Mugnolo
	and Pl\"umer to the full range $1<p<\infty$.
	
	
\end{abstract}

\section{Introduction}

Torsional rigidity is a fundamental quantity in elasticity theory,
originating in the work of Saint-Venant on the torsion of prismatic rods
\cite{SV}; its variational formulation was later established by P\'olya
\cite{Polya, pol3}. Beyond its origins in elasticity, torsional rigidity has been studied in spectral geometry and shape optimization, owing to its close connections with isoperimetric inequalities, eigenvalue estimates, and other geometric functionals associated with elliptic operators. For a bounded domain $\Omega\subset\mathbb{R}^d$, the
torsional rigidity associated with the Dirichlet Laplacian is defined by
\[
T(\Omega)=\int_\Omega v\,dx,
\]
where $v$ is the unique positive solution of
\[
\begin{cases}
	-\Delta v =1, & x\in\Omega,\\
	v=0, & x\in\partial\Omega.
\end{cases}
\]
Equivalently, it admits the variational characterization
\[
T(\Omega)
=
\sup_{u\in H_0^1(\Omega)\setminus\{0\}}
\frac{\|u\|_{L^1(\Omega)}^2}
{\|\nabla u\|_{L^2(\Omega)}^2}.
\]
The interplay between torsional rigidity, the first Dirichlet eigenvalue,
and isoperimetric inequalities has been extensively studied; see, for
instance, the monograph of P\'olya and Szeg\H{o} \cite{PS}. Among the
classical results in this direction, two are particularly relevant for the
present paper. The Saint-Venant inequality identifies the ball as the
maximizer of torsional rigidity among domains of prescribed measure,
whereas the Kohler--Jobin inequality \cite{KJ} gives a sharp
scale-invariant lower bound for the product of the first Dirichlet
eigenvalue and a suitable power of the torsional rigidity.

These inequalities also admit nonlinear counterparts for the
$p$-Laplacian. Fragal\`a, Gazzola, and Lamboley \cite{FGL13} established
sharp bounds for the $p$-torsional rigidity of convex planar domains,
while Kohler--Jobin  developed a rearrangement approach for the
$p$-Laplacian, lastly extended by Brasco \cite{Brasco}, leading to a $p$-Kohler--Jobin inequality together with a
variational characterization of the $p$-torsional rigidity.

Over the past decades, compact metric graphs have become a well-established framework for the study of differential operators on network-like structures, with applications ranging from quantum mechanics to wave propagation in thin structures. Their spectral and variational theory has developed rapidly, revealing deep connections between the geometry of a graph and the analytic properties of the associated operators; see, for instance, \cite{BK,kuch1, kuch2}.
Among the various spectral and variational quantities that have recently attracted attention in this setting is torsional rigidity.

In this direction, the linear case $p=2$ was systematically investigated
by Mugnolo and Pl\"umer \cite{MuPl}, who established a variational
characterization of the torsional rigidity together with graph analogues
of the Saint-Venant, P\'olya--Szeg\H{o}, and Kohler--Jobin inequalities.
In recent years, torsional rigidity has also been investigated in
different graph settings. For combinatorial graphs, Bifulco and Mugnolo
\cite{BM25} introduced and studied the $p$-torsional rigidity associated
with the discrete $p$-Laplacian. In a different direction, \"Ozcan and
T\"aufer \cite{OT24} studied torsional rigidity on metric graphs with
$\delta$-vertex conditions. In parallel, the spectral theory of the $p$-Laplacian on metric graphs
has been developed by Del Pezzo and Rossi \cite{DR}, who established the
existence, simplicity, and basic properties of the first eigenvalue
under Dirichlet--Kirchhoff conditions.

Although rearrangement methods are well established both for Euclidean
domains and for the linear Laplacian on metric graphs, their extension to
the nonlinear setting is far from straightforward. Indeed, the nonlinear
structure of the $p$-Laplacian alters both the variational framework and
the level-set analysis, while the Kirchhoff transmission conditions must
be reformulated in terms of nonlinear fluxes. Consequently, the arguments
developed in the linear case cannot be transferred directly. In this paper
we overcome these difficulties and establish nonlinear counterparts of the
Saint-Venant and Kohler--Jobin inequalities for the $p$-Laplacian on
compact metric graphs.

We consider the $p$-Laplacian on compact metric graphs with Dirichlet
conditions imposed on a nonempty set $\mathcal V^D$ of vertices of degree
one and natural nonlinear Kirchhoff conditions at all remaining vertices,
collected in the set $\mathcal V^N:=\mathcal V\setminus\mathcal V^D$. The
latter require the conservation of the nonlinear flux at every vertex of
$\mathcal V^N$ and therefore constitute the natural analogue of the
classical Kirchhoff conditions for the linear Laplacian; at vertices of
degree one they reduce to Neumann conditions.
Let
\[
\WpD:=\{u\in W^{1,p}(\mathcal G):u|_{\mathcal V^D}=0\}
\]
denote the natural energy space.
The $p$-torsion function $w_p\in\WpD$ is the unique weak
solution of
\[
-\Delta_p w_p =1,
\]
subject to
\[
w_p|_{\mathcal V^D}=0,
\qquad
\sum_{e\in\mathcal E_v}
|w_{p,e}'(v)|^{p-2}
\partial_{\nu_e}w_{p,e}(v)=0,
\quad
v\in\mathcal V^N,
\]
where $\mathcal E_v$ denotes the set of edges incident with $v$ and
$\partial_{\nu_e}$ the outward derivative along the edge $e$.
The corresponding $p$-torsional rigidity is defined by
$$T_p(\mathcal{G},\mathcal{V}^D):=\left(\int_{\mathcal{G}} w_p\,dx\right)^{p-1},$$
which reduces to the classical torsional rigidity for $p=2$.
We prove the following nonlinear variational
characterization:
\[
T_p(\mathcal{G},\mathcal{V}^D)
\;=\;
\sup_{u\in\WpD,\,u\not\equiv 0}
\frac{\bigl(\int_{\mathcal{G}}\abs{u}\,dx\bigr)^p}{\int_{\mathcal{G}}\abs{u'}^p\,dx}.
\]
The variational characterization also yields several structural properties
of the $p$-torsional rigidity. In particular, we establish its natural
scaling law and show that the normalized quantity
\(
p\mapsto
|\mathcal G|^{1/p}
T_p(\mathcal G,\mathcal V^D)^{1/p}
\)
is non-increasing on $(1,\infty)$.

The proof of the $p$-Saint-Venant inequality requires a nonlinear
adaptation of the rearrangement method developed in the linear case.
While the distribution function is treated in a similar spirit, the coarea
analysis naturally involves the nonlinear quantity $|u'|^{p-1}$, replacing
the linear gradient appearing in the case $p=2$, and the key step is a
rearrangement inequality showing that the $p$-Dirichlet energy does not
increase under decreasing rearrangement. This allows us to compare the
$p$-torsional rigidity of an arbitrary compact metric graph with that of
the interval of the same total length with a Dirichlet
endpoint.

The proof of the $p$-Kohler--Jobin inequality is more delicate. The key
idea, going back to Kohler-Jobin and improved for the $p$-Laplacian by
Brasco \cite{Brasco}, is to perform the rearrangement at the level of the
first eigenfunction: its level sets give rise to a modified torsional rigidity, which is used to calibrate the length of a one-dimensional
comparison interval in such a way that both the first eigenvalue and the
torsional rigidity of that interval are controlled by the corresponding
quantities of the graph. The nonlinear nature of the $p$-Laplacian
prevents a direct adaptation of the linear graph argument and requires a
refined level-set analysis combining  rearrangement method with
the metric graph framework developed in \cite{MuPl}.

The paper is organized as follows. In
Section~\ref{sec:setup} we introduce the necessary background on compact
metric graphs, the Sobolev spaces, and
the variational formulation of the $p$-Laplacian with Dirichlet and
$p$-Kirchhoff vertex conditions. Section~\ref{sec:torsion} establishes
the existence, uniqueness, and positivity of the $p$-torsion function,
introduces the $p$-torsional rigidity together with its variational
characterization, and studies its scaling properties and monotonicity
with respect to $p$. In
Section~\ref{sec:saintvenant} we derive the $p$-Saint-Venant inequality
and characterize the equality case by means of a rearrangement
argument. Finally, Section~\ref{sec:spectral} is devoted to the spectral
applications, where we establish the $p$-P\'olya--Szeg\H{o} and sharp
$p$-Kohler--Jobin inequalities.
\section{Preliminaries on metric graphs}
\label{sec:setup}

We begin by fixing the terminology for compact metric graphs; for a
comprehensive introduction we refer to~\cite{BK}.

A \emph{compact metric graph} $\mathcal G$ is obtained from a finite
combinatorial graph by identifying the endpoints of compact intervals
$[0,\ell_e]$, $\ell_e>0$, according to the underlying graph structure.
We denote by $\mathcal{V}$ and $\mathcal{E}$ the sets of vertices and edges of
$\mathcal G$, respectively, and
\[
|\mathcal G|:=\sum_{e\in \mathcal{E}}\ell_e
\]
for the total length of $\mathcal G$.
For each vertex $v\in \mathcal{V}$, we let
\(
\mathcal{E}_v:=\{e\in \mathcal{E}:\,v\in e\}
\)
be the set of edges incident with $v$, and define
$\deg(v):=|\mathcal{E}_v|$. We say that $\mathcal G$ is \emph{connected} if
every pair of points in $\mathcal G$ can be joined by a continuous path
contained in $\mathcal G$.

For a function $u:\mathcal G\to\mathbb C$, we write $u_e$ for its restriction
to the edge $e\in\mathcal E$. If a vertex $v$ is incident with an edge $e$,
then
\(
\partial_{\nu_e}u_e(v)
\)
denotes the derivative of $u_e$ at $v$ taken in the outward direction along
the edge $e$.

For $1\le p\le\infty$, we define
\[
L^p(\mathcal G)
:=
\bigoplus_{e\in\mathcal E}L^p(0,\ell_e),
\]
equipped with the norm
\[
\|u\|_{L^p(\mathcal G)}
=
\begin{cases}
	\displaystyle
	\left(
	\sum_{e\in\mathcal E}
	\int_0^{\ell_e}|u_e(x)|^p\,dx
	\right)^{1/p},
	& 1\le p<\infty,\\[1.2em]
	\displaystyle
	\max_{e\in\mathcal E}
	\operatorname*{ess\,sup}_{x\in(0,\ell_e)}
	|u_e(x)|,
	& p=\infty.
\end{cases}
\]
The following assumption holds throughout the paper.

\begin{assumption}\label{ass:main}
	The metric graph $\mathcal G$ is connected and has finitely many edges.
	The \emph{Dirichlet set} $\mathcal{V}^D\subseteq \mathcal{V}$ is nonempty and consists only of
	vertices of degree one. We write $\mathcal{V}^N:=\mathcal{V}\setminus\mathcal{V}^D$ and refer to the vertices
	in $\mathcal{V}^N$ as \emph{Kirchhoff vertices}.
\end{assumption}

The assumption that every Dirichlet vertex has degree one is made for
convenience rather than necessity. Indeed, a Dirichlet condition
disconnects the incident edges, so any Dirichlet vertex of higher degree
can be replaced by several degree-one Dirichlet vertices without changing
any of the qquantities studied below.

Throughout the paper we fix the
H\"older conjugate of the exponent $1<p<\infty$ by $q=p/(p-1)$.
We define
\[
W^{1,p}(\mathcal G):=
\bigl\{u:\mathcal G\to\mathbb R:\,
u|_e\in W^{1,p}(0,\ell_e)\ \forall e\in \mathcal{E},
\ u\text{ is continuous at every }v\in \mathcal{V}
\bigr\},
\]
equipped with the norm
\(
\|u\|_{W^{1,p}(\mathcal G)}^p
:=
\|u\|_{L^p(\mathcal{G})}^p
+
\|u'\|_{L^p(\mathcal{G})}^p
.
\)
Since $p>1$, the embedding
$W^{1,p}(0,\ell_e)\hookrightarrow C([0,\ell_e])$
holds on every edge, so the continuity condition at the vertices is
well defined and every $u\in W^{1,p}(\mathcal G)$ is in fact continuous on
all of $\mathcal G$. Furthermore, by \cite[Theorem~2.2]{DR}, the injection
\(
W^{1,p}(\mathcal G)\hookrightarrow L^r(\mathcal G)
\)
is compact for every $1\le r\le\infty$.

We introduce the closed subspace
\[
\WpD
:=
\{u\in W^{1,p}(\mathcal G):u(v)=0 \text{ for all }v\in \mathcal{V}^D\},
\]
which will serve as the energy space for all variational problems in
this paper. Since $W^{1,p}(\mathcal G)$ is isomorphic to a closed
subspace of the reflexive space
$\prod_{e\in\mathcal E}W^{1,p}(0,\ell_e)$, both $W^{1,p}(\mathcal G)$
and its closed subspace $\WpD$ are reflexive.

The Dirichlet condition, combined with the connectedness of
$\mathcal G$, forces functions in $\WpD$ to be small in a quantitative
sense. This is the content of the following Poincar\'e inequality, which
we will use repeatedly: for every $u\in\WpD$,
\[
\|u\|_{L^\infty(\mathcal G)}
\le
|\mathcal G|^{1/q}\,\|u'\|_{L^p(\mathcal G)},
\qquad\text{and hence}\qquad
\|u\|_{L^p(\mathcal G)}
\le
C_P\,\|u'\|_{L^p(\mathcal G)}
\quad\text{with }C_P:=|\mathcal G|.
\]
Indeed, given $x\in\mathcal G$ and $v_0\in\mathcal V^D$, connectedness
provides an injective path $P\subseteq\mathcal G$ from $v_0$ to $x$, of
length at most $|\mathcal G|$, and
\[
|u(x)|=|u(x)-u(v_0)|
\le\int_P|u'|\,dy
\le|\mathcal G|^{1/q}\|u'\|_{L^p(\mathcal G)}
\]
by H\"older's inequality; the second estimate follows since
$\|u\|_{L^p(\mathcal G)}\le|\mathcal G|^{1/p}\|u\|_{L^\infty(\mathcal G)}$.

The energy functional associated to the $p$-Laplacian on $\mathcal G$ is defined by
\begin{equation}\label{eq:Fp}
	\mathcal F_p(u)
	:=
	\frac1p\int_{\mathcal G}|u'|^p\,dx,
	\qquad
	u\in\WpD.
\end{equation}
Since $1<p<\infty$, the map $t\mapsto |t|^p$ is strictly convex, and therefore
$\mathcal F_p$ is strictly convex on $\WpD$; moreover, the Poincar\'e
inequality shows that $\mathcal F_p$ is coercive. Given
$f\in L^1(\mathcal G)$, a function
$u\in\WpD$ is called a \emph{weak solution} of
\(
-\Delta_pu=f
\)
if
\begin{equation}\label{eq:plap_duality}
	\int_{\mathcal G}|u'|^{p-2}u'v'\,dx
	=
	\int_{\mathcal G}fv\,dx,
	\qquad
	v\in\WpD.
\end{equation}
The weak formulation, which follows from an integration by parts, is equivalent to
\begin{alignat}{2}
	-(|u_e'|^{p-2}u_e')'
	&=f|_e,
	&&\qquad e\in \mathcal{E},\notag\\
	u(v)&=0,
	&&\qquad v\in \mathcal{V}^D,\notag\\
	\sum_{e\in \mathcal{E}_v}|u_e'(v)|^{p-2}\partial_{\nu_e}u_e(v)
	&=0,
	&&\qquad v\in \mathcal{V}^N.
	\label{eq:Kirch_setup}
\end{alignat}
Condition~\eqref{eq:Kirch_setup} is the natural $p$-Kirchhoff vertex
condition. It expresses the conservation of the nonlinear flux at every
vertex of $\mathcal V^N$ and reduces to the classical Kirchhoff condition
when $p=2$.
\section{\texorpdfstring{$p$}{p}-Torsion Function
	and \texorpdfstring{$p$}{p}-Torsional Rigidity}
\label{sec:torsion}

In this section, we introduce the central object of the paper: the
$p$-torsion function of a compact metric graph. We derive its
existence, uniqueness, and strict positivity away from the Dirichlet
vertices. The total mass of the torsion
function then leads us to the notion of $p$-torsional rigidity and to
its variational characterization.
\subsection{The \texorpdfstring{$p$}{p}-torsion function}

The following definition is the natural nonlinear analogue of the
classical torsion problem.

\begin{definition}\label{def:torsionfunction}
	Let $\mathcal G$ be a connected compact metric graph, 
	$\mathcal{V}^D\subseteq \mathcal{V}$ be a nonempty set of degree-one vertices, and 
	$1<p<\infty$.	A function $w_p\in\WpD$ is called the \emph{$p$-torsion function} of
	$\mathcal G$ if
	\begin{alignat}{2}
		-(|w_{p,e}'|^{p-2}w_{p,e}')' &= 1
		&&\qquad \text{on each } e\in \mathcal{E}, \label{eq:pLap}\\
		w_p(v) &=0
		&&\qquad v\in \mathcal{V}^D, \label{eq:Dir}\\
		\sum_{e\in \mathcal{E}_v}|w_{p,e}'(v)|^{p-2}\partial_{\nu_e}w_{p,e}(v) &=0
		&&\qquad v\in\mathcal{V}^N. \label{eq:Kirch}
	\end{alignat}
\end{definition}

Our first task is to show that this problem is well posed. The proof
relies on the direct method of the calculus of variations, applied to a
functional whose Euler--Lagrange equation is precisely the torsion
problem.

\begin{theorem}\label{thm:exist}
	Let $\mathcal G$ be a connected compact metric graph, 
	$\mathcal{V}^D\subseteq\mathcal{V}$ be a nonempty set of degree-one vertices, and 
	$1<p<\infty$.
	Then there exists a unique $p$-torsion function
	$w_p\in W^{1,p}_{\mathcal{V}^D}(\mathcal G)$.
\end{theorem}
\begin{proof}
	Consider the functional
	\begin{equation}\label{eq:Jp}
		J_p(u):=\frac1p\int_{\mathcal G}|u'|^p\,dx-\int_{\mathcal G}u\,dx,
		\qquad u\in\WpD .
	\end{equation}
	By H\"older's inequality and the Poincar\'e inequality,
	\[
	\int_{\mathcal G}u\,dx
	\le |\mathcal G|^{1/q}\|u\|_{L^p(\mathcal G)}
	\le |\mathcal G|^{1/q}C_P\|u'\|_{L^p(\mathcal G)} .
	\]
	Thus, with $C_0:=|\mathcal G|^{1/q}C_P$,
	\begin{equation}\label{eq:Jp_lower}
		J_p(u)\ge
		\frac1p\|u'\|_{L^p(\mathcal G)}^p
		-C_0\|u'\|_{L^p(\mathcal G)} .
	\end{equation}
	As $p>1$, the right-hand side is bounded from below and tends to
	$+\infty$ as $\|u'\|_{L^p(\mathcal G)}\to\infty$; together with the
	Poincar\'e inequality, this shows that $J_p$ is bounded from below and
	coercive on $\WpD$.
	
	Let $(u_n)\subset\WpD$ be a minimizing sequence. By coercivity, $(u_n)$ is
	bounded in $\WpD$. Since $\WpD$ is reflexive, after passing to a subsequence
	we may assume that
	\(
	u_n\rightharpoonup w_p
	\text{ weakly in }\WpD
	\)
	for some $w_p\in\WpD$. The functional
	$u\mapsto \frac1p\int_{\mathcal G}|u'|^p\,dx$ is convex and strongly
	continuous, hence weakly lower
	semicontinuous, while $u\mapsto\int_{\mathcal G}u\,dx$ is weakly continuous.
	Hence
	\[
	J_p(w_p)\le \liminf_{n\to\infty}J_p(u_n)=\inf_{\WpD}J_p,
	\]
	so $w_p$ is a minimizer.
	
	Uniqueness of the minimizer is a consequence of convexity: if
	$u,v\in\WpD$ are two minimizers, then the strict convexity of
	$t\mapsto |t|^p$ implies strict convexity of the first term in $J_p$
	unless $u'=v'$ a.e.\ on $\mathcal G$. Since the second term is linear,
	minimality forces $u'=v'$ a.e., and therefore $u-v$ is constant on the
	connected graph $\mathcal G$. As $\mathcal{V}^D\ne\emptyset$ and both
	functions vanish on $\mathcal{V}^D$, this constant is zero.
	
	We next identify the Euler--Lagrange equation. Since $w_p$ minimizes
	$J_p$, for every $\phi\in\WpD$ we have
	\begin{equation}\label{eq:weak_torsion}
		\int_{\mathcal G}|w_p'|^{p-2}w_p'\phi'\,dx
		=
		\int_{\mathcal G}\phi\,dx .
	\end{equation}
	Taking test functions supported in the interior of a single edge,
	\eqref{eq:weak_torsion} reduces to
	\[
	\int_e |w_{p,e}'|^{p-2}w_{p,e}'\,\phi'
	=
	\int_e\phi,
	\qquad
	\phi\in C_c^\infty(e^\circ),
	\]
	so that
	\(
	-(|w_{p,e}'|^{p-2}w_{p,e}')'=1
	\)
	holds in the distributional sense on $e$. Consequently, the flux
	\(
	A_e:=|w_{p,e}'|^{p-2}w_{p,e}'
	\)
	satisfies $-A_e'=1$, whence
	$A_e(x)=c_e-x$ for some constant $c_e$. In particular,
	$A_e$ extends continuously to the endpoints of $e$, and the boundary
	fluxes appearing at the vertices are therefore well defined.
	
	Now, integrating by parts in \eqref{eq:weak_torsion} over all edges and
	using \eqref{eq:pLap}, the interior terms cancel and one obtains
	\[
	\sum_{v\in \mathcal{V}^N}\phi(v)
	\sum_{e\in \mathcal{E}_v}
	|w_{p,e}'(v)|^{p-2}\partial_{\nu_e}w_{p,e}(v)=0 .
	\]
	Since the values $\phi(v)$, $v\in \mathcal{V}^N$, can be chosen
	independently, the $p$-Kirchhoff condition \eqref{eq:Kirch}
	follows. The Dirichlet condition is already encoded in the choice
	$w_p\in\WpD$. Hence the minimizer is a $p$-torsion function in the
	sense of Definition~\ref{def:torsionfunction}.
	
	Conversely, suppose $w\in\WpD$ satisfies
	\eqref{eq:pLap}--\eqref{eq:Kirch}. Multiplying \eqref{eq:pLap} by an
	arbitrary $\phi\in\WpD$ and integrating by parts edge by edge, the
	boundary terms at the Dirichlet vertices vanish because $\phi$ does,
	while those at the Kirchhoff vertices cancel by \eqref{eq:Kirch}; thus
	$w$ solves \eqref{eq:weak_torsion}, i.e., $w$ is a critical point of
	the convex functional $J_p$. Critical points of convex functionals are
	global minimizers, so $w$ coincides with the unique minimizer
	constructed above. This proves uniqueness within the class of
	Definition~\ref{def:torsionfunction} and completes the proof.
\end{proof}

For $p=2$, this reduces to the usual torsion problem
$-\Delta w_2=1$ with Dirichlet conditions on $\mathcal{V}^D$ and Kirchhoff conditions on
$\mathcal{V}^N$ (\cite{MuPl}).

The next result shows that the torsion function is strictly positive
away from the Dirichlet boundary. Beyond its intrinsic interest, this
property is what allows the torsion function itself to compete in the
variational problems studied below.
\begin{proposition}
	\label{prop:positivity}
	Let $\mathcal G$ be a connected compact metric graph, 
	$\mathcal{V}^D \subseteq \mathcal{V}$ be a nonempty set of degree-one vertices, and 
	$1<p<\infty$.	The $p$-torsion function is positive on
	$\mathcal G\setminus \mathcal{V}^D$.
\end{proposition}

\begin{proof}
	We first show that $w_p$ is nonnegative. Let
	\(
	w_p^-:=\max\{-w_p,0\}.
	\)
	The map $t\mapsto\max\{-t,0\}$ is Lipschitz, so
	$w_p^-\in\WpD$ and
	\[
	(w_p^-)'=
	\begin{cases}
		-w_p' & \text{a.e.\ on }\{w_p<0\},\\
		0 & \text{a.e.\ on }\{w_p\ge0\}.
	\end{cases}
	\]
	Using $w_p^-$ as a test function in the weak formulation
	\eqref{eq:weak_torsion}, we obtain
	\[
	\int_{\mathcal G}|w_p'|^{p-2}w_p'(w_p^-)'dx
	=
	\int_{\mathcal G}w_p^-\,dx,
	\]
	that is,
	\[
	-\int_{\{w_p<0\}}|w_p'|^p\,dx
	=
	\int_{\mathcal G}w_p^-\,dx.
	\]
	The left-hand side is nonpositive, whereas the right-hand side is
	nonnegative; both sides must therefore vanish, and
	$w_p^-=0$ almost everywhere. Since $w_p$ is continuous on
	$\mathcal G$, we conclude that
	\(
	w_p\ge0 \text{ on }\mathcal G.
	\)
	
	To show strict positivity, we apply the strong
	maximum principle of V\'azquez \cite[Theorem~5]{Vazquez} on each edge
	separately. On every edge we have
	$\Delta_p w_p=-1\le0$ and $w_p\ge0$; moreover, $w_p$ cannot vanish
	identically on any edge, as this would contradict
	$-(|w_{p,e}'|^{p-2}w_{p,e}')'=1$. Hence $w_p$ is positive in the
	interior of every edge.
	
	Suppose now that $w_p(v_0)=0$ for some $v_0\in \mathcal{V}^N$. Viewing
	each incident edge $e\in\mathcal E_{v_0}$ as a one-dimensional domain
	with boundary point $v_0$, the boundary point lemma in
	\cite[Theorem~5]{Vazquez} yields
	\(
	\partial_{\nu_e}w_{p,e}(v_0)>0,
\forall e\in \mathcal{E}_{v_0},
	\)
which gives
	\[
	\sum_{e\in \mathcal{E}_{v_0}}
	|w_{p,e}'(v_0)|^{p-2}
	\partial_{\nu_e}w_{p,e}(v_0)
	>0,
	\]
	contradicting \eqref{eq:Kirch}.
\end{proof}
\subsection{\texorpdfstring{$p$}{p}-Torsional rigidity and its variational characterization}
Having proved the existence and positivity of the torsion function,
we can now attach to it a single number measuring the overall response
of the graph to a uniform unit load.
\begin{definition}\label{def:Tp}
	Let $\mathcal G$ be a connected compact metric graph,
	$\mathcal{V}^D\subseteq \mathcal{V}$ be a nonempty set of degree-one vertices, and 
	$1<p<\infty$. 
	The \emph{$p$-torsional rigidity} of $\mathcal G$ is defined by
	\begin{equation}
		\label{eq:torsion_def}
		T_p(\mathcal G,\mathcal{V}^D)
		:=
		\left(
		\int_{\mathcal G}w_p\,dx
		\right)^{p-1},
	\end{equation}
	where $w_p$ denotes the $p$-torsion function.
\end{definition}

The exponent $p-1$ may look arbitrary at first sight; it is chosen so
that $T_p$ obeys the same scaling law as the classical torsional
rigidity (see Lemma~\ref{lem:scaling} below, and \cite[Proposition 4.1 (6)]{MuPl}) and reduces to it for
$p=2$. The definition through the torsion function has the drawback of
being implicit; the following variational characterization removes this
drawback and will be our main working tool.

\begin{proposition}
	\label{prop:char}
	Let $\mathcal G$ be a connected compact metric graph,
	let $\mathcal V^D\subseteq\mathcal V$ be a nonempty set of degree-one
	vertices, and let $1<p<\infty$. Then the $p$-torsional rigidity admits
	the variational characterization
	\begin{equation}
		\label{eq:varchar}
		T_p(\mathcal G,\mathcal V^D)
		=
		\sup_{\substack{u\in\WpD\\u\not\equiv0}}
		\frac{\left(\int_{\mathcal G}|u|\,dx\right)^p}
		{\int_{\mathcal G}|u'|^p\,dx}.
	\end{equation}
	Moreover, a function $u\in\WpD\setminus\{0\}$ attains the supremum
	in~\eqref{eq:varchar} if and only if it is a nonzero scalar multiple
	of the $p$-torsion function $w_p$.
\end{proposition}

\begin{proof}
	Taking $\phi=w_p$ as a test function in the weak formulation
	\eqref{eq:weak_torsion}, we obtain the energy identity
	\begin{equation}
		\label{eq:torsion_energy}
		\int_{\mathcal G}|w_p'|^p\,dx
		=
		\int_{\mathcal G}w_p\,dx.
	\end{equation}
	In particular,
	\[
	\frac{\left(\int_{\mathcal G}w_p\,dx\right)^p}
	{\int_{\mathcal G}|w_p'|^p\,dx}
	=
	\left(\int_{\mathcal G}w_p\,dx\right)^{p-1}
	=
	T_p(\mathcal G,\mathcal V^D),
	\]
	so the supremum in~\eqref{eq:varchar} is at least
	$T_p(\mathcal G,\mathcal V^D)$.
	Conversely, let $u\in\WpD\setminus\{0\}$. Since
	$|u|\in\WpD$ and
	\(
	\bigl|(|u|)'\bigr|=|u'|
	\text{ a.e. on }\mathcal G,
	\)
	the variational quotient is unchanged when $u$ is replaced by $|u|$.
	We may therefore assume that $u\geq0$. Using $u$ as a test function in
	\eqref{eq:weak_torsion} gives
	\[
	\int_{\mathcal G}u\,dx
	=
	\int_{\mathcal G}|w_p'|^{p-2}w_p'u'\,dx.
	\]
	By H\"older's inequality,
	\[
	\int_{\mathcal G}u\,dx
	\leq
	\left(
	\int_{\mathcal G}
	\bigl||w_p'|^{p-2}w_p'\bigr|^q\,dx
	\right)^{1/q}
	\left(
	\int_{\mathcal G}|u'|^p\,dx
	\right)^{1/p}.
	\]
	Since $(p-1)q=p$, this becomes
	\[
	\int_{\mathcal G}u\,dx
	\leq
	\left(
	\int_{\mathcal G}|w_p'|^p\,dx
	\right)^{(p-1)/p}
	\left(
	\int_{\mathcal G}|u'|^p\,dx
	\right)^{1/p}.
	\]
	Raising both sides to the power $p$, dividing by
	$\int_{\mathcal G}|u'|^p\,dx$, and using
	\eqref{eq:torsion_energy}, we find
	\[
	\frac{\left(\int_{\mathcal G}u\,dx\right)^p}
	{\int_{\mathcal G}|u'|^p\,dx}
	\leq
	\left(
	\int_{\mathcal G}|w_p'|^p\,dx
	\right)^{p-1}
	=
	\left(
	\int_{\mathcal G}w_p\,dx
	\right)^{p-1}
	=
	T_p(\mathcal G,\mathcal V^D).
	\]
	Taking the supremum over all nonzero $u\in\WpD$ proves
	\eqref{eq:varchar}.
	
	It remains to characterize the maximizers. Since the quotient in
	\eqref{eq:varchar} is homogeneous of degree zero, every nonzero scalar
	multiple of $w_p$ attains the supremum.
Conversely, let $u$ be a maximizer. Then equality holds in Hölder's
inequality, and hence
\[
u'
=
c\,|w_p'|^{(p-1)(q-2)+p-2}w_p'
=
cw_p'
\]
for some constant $c>0$, since $(p-1)(q-2)+p-2=0$.
Hence $u-cw_p$ is constant on $\mathcal G$. As both functions vanish on
$\mathcal V^D$, this constant is zero, and therefore
\(
u=cw_p.
\)
This proves that every maximizer is a scalar multiple of the
$p$-torsion function.
\end{proof}

Formula \eqref{eq:varchar} is the metric graph analogue of the classical
variational characterization of torsional rigidity due to
P\'olya~\cite{Polya}; its nonlinear counterpart on Euclidean domains was
established by Brasco~\cite{Brasco}. For $p=2$, it reduces to
\cite[Eq.~(2.4)]{MuPl}.

\subsection{Basic properties of the \texorpdfstring{$p$}{p}-torsional rigidity}

The variational characterization established in Proposition~\ref{prop:char}
immediately yields several structural properties of the $p$-torsional
rigidity. We begin with its scaling behavior, which in particular
justifies the exponent chosen in Definition~\ref{def:Tp}.

\begin{lemma}
	\label{lem:scaling}
	Let $\mathcal G$ be a connected compact metric graph, let
	$\mathcal{V}^D\subseteq \mathcal{V}$ be a nonempty set of degree-one vertices, and
	$1<p<\infty$.	Let $c\mathcal G$ denote the metric graph obtained from $\mathcal G$ by
	multiplying every edge length by a factor $c>0$, while keeping the same
	Dirichlet set $\mathcal{V}^D$. Then
	\[
	T_p(c\mathcal G,\mathcal{V}^D)=c^{2p-1}	T_p(\mathcal G,\mathcal{V}^D).
	\]
\end{lemma}

\begin{proof}
	For $u\in\WpD$, define $u_c$ on $c\mathcal G$ by
	\(
	u_c(x):=u(x/c);
	\)
	the map $u\mapsto u_c$ is a bijection of $\WpD$ onto
	$W^{1,p}_{\mathcal V^D}(c\mathcal G)$.
	A change of variables gives
	\[
	\int_{c\mathcal G}|u_c|\,dx
	=
	c\int_{\mathcal G}|u|\,dx,
	\qquad
	\int_{c\mathcal G}|u_c'|^p\,dx
	=
	c^{1-p}\int_{\mathcal G}|u'|^p\,dx,
	\]
	and therefore
	\[
	\frac{\left(\int_{c\mathcal G}|u_c|\,dx\right)^p}
	{\int_{c\mathcal G}|u_c'|^p\,dx}
	=
	c^{2p-1}
	\frac{\left(\int_{\mathcal G}|u|\,dx\right)^p}
	{\int_{\mathcal G}|u'|^p\,dx}.
	\]
	Taking the supremum  and using
	Proposition~\ref{prop:char} give the result.
\end{proof}

Our second result describes how the torsional rigidity depends on the
exponent $p$ itself: once suitably normalized by the total length, it
can only decrease as $p$ grows.

\begin{proposition}
	\label{prop:mono}
	Let $\mathcal G$ be a connected compact metric graph with  $|\mathcal G|=L$,
	$\mathcal{V}^D\subseteq \mathcal{V}$ be a nonempty set of degree-one vertices, and 
	$1<r<p<\infty$. Then
	\[
	T_p(\mathcal G,\mathcal{V}^D)^{1/p}
	\le
	L^{1/r-1/p}
	T_r(\mathcal G,\mathcal{V}^D)^{1/r}.
	\]
	Equivalently, the map
	\(
	p\mapsto
	L^{1/p}	T_p(\mathcal G,\mathcal{V}^D)^{1/p}
	\)
	is non-increasing on $(1,\infty)$.
\end{proposition}

\begin{proof}
	Let $u\in\WpD$ with $u\not\equiv0$. By H\"older's inequality,
	\(
	\|u'\|_{L^r(\mathcal G)}
	\le
	L^{1/r-1/p}\,
	\|u'\|_{L^p(\mathcal G)},
	\)
	or equivalently,
	\[
	\int_{\mathcal G}|u'|^p\,dx
	\ge
	L^{1-p/r}\,
	\|u'\|_{L^r(\mathcal G)}^p.
	\]
	Substituting this estimate into the denominator of the variational quotient
	gives
	\[
	\frac{\left(\int_{\mathcal G}|u|\,dx\right)^p}
	{\int_{\mathcal G}|u'|^p\,dx}
	\le
	L^{p/r-1}
	\left(
	\frac{\int_{\mathcal G}|u|\,dx}
	{\|u'\|_{L^r(\mathcal G)}}
	\right)^p.
	\]
	Taking the supremum over $u\in\WpD$ and applying
	Proposition~\ref{prop:char} for both the exponent $p$ and
 $r$  yields
	\(
L^{1/p}	T_p(\mathcal G,\mathcal{V}^D)^{1/p}
	\le
	L^{1/r}
	T_r(\mathcal G,\mathcal{V}^D)^{1/r}.
	\)
\end{proof}
\section{The Saint-Venant inequality}
\label{sec:saintvenant}

The main result of this section is a Saint-Venant type inequality for
the $p$-torsional rigidity: among all compact metric graphs of given
total length, the torsional rigidity is largest when the graph is an interval. The guiding idea is the following. Starting from the
torsion function of an arbitrary graph, we produce, by monotone
rearrangement, a competitor on an interval of the same total length,
and we show that this passage preserves the total mass while it can
only decrease the $p$-Dirichlet energy, the loss of energy being
caused precisely by branching and by cycles. Since the extremal
configuration is an interval with a single Dirichlet endpoint, we begin
by computing its torsional rigidity explicitly.


Let $J=[0,L]$ denote the interval of length $L$.

\begin{example}
	\label{ex:J0}
		\begin{figure}\label{interval graph}
		\centering
		\begin{tikzpicture}
			\node[draw, circle, inner sep=1pt, fill, label=above:0] (A) at (0, 0) {};
			\node[draw, circle, inner sep=1pt, fill, label=above:$L$] (B) at (4, 0) {};
			\draw[ ] (A) -- (B);
		\end{tikzpicture}
		\caption{The interval graph of length $L$}
	\end{figure}
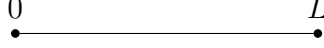
	Let $\mathcal{V}^D=\{L\}$, so that $J$ is equipped with a Neumann condition at $0$ and a
	Dirichlet condition at $L$. 
	Let $w_p$ denote the torsion function on $J$. Introducing the $p$-flux
	\(
	A:=|w_p'|^{p-2}w_p',
	\)
	the torsion equation becomes
	\(
	-A'=1.
	\)
	Since the Neumann condition is equivalent to $A(0)=0$, we obtain
	\(
	A(x)=-x.
	\)
	Consequently,
	\(
	w_p'(x)=-x^{q-1},
	\)
	and integrating together with the Dirichlet condition $w_p(L)=0$ yields
	\(
	w_p(x)
	=
	\frac{L^q-x^q}{q}.
	\)
Then the torsional rigidity is found as
	\begin{equation}
		\label{eq:J0}
		T_p(J,\{L\})
		= 	\int_0^Lw_p(x)\,dx
		=
		\frac{L^{q+1}}{q+1}.
	\end{equation}

\end{example}

We can now state the main result of this section. The proof adapts the
rearrangement argument of \cite[Section~5.1]{AST17} to the nonlinear
setting $1<p<\infty$.

\begin{theorem}
	\label{thm:SV}
	Let $\mathcal G$ be a connected compact metric graph with $|\mathcal G|=L$,
	let $\mathcal{V}^D\subseteq \mathcal{V}$ be a nonempty set of degree-one
	vertices, and let $1<p<\infty$. Then
	\begin{equation}
		\label{eq:SV}
		T_p(\mathcal G,\mathcal{V}^D)
		\le
		\Bigl(\frac{L^{q+1}}{q+1}\Bigr)^{p-1}
		=
		T_p(J,\{L\}),
	\end{equation}
	with equality if and only if
	$\mathcal G$ is an interval of length $L$ with a single Dirichlet endpoint.
\end{theorem}

\begin{proof}
	Let $w_p\in\WpD$ be the $p$-torsion function on
	$(\mathcal G,\mathcal{V}^D)$. By Proposition~\ref{prop:positivity},
	$w_p\ge0$ on $\mathcal G$ and $w_p\not\equiv0$; $w_p$ is edgewise $C^1$ and piecewise
	strictly monotone, with at most one critical point per edge. In
	particular, the set consisting of the critical values of $w_p$ and of
	its values at the vertices is finite. 
 For  $t\in(0,\|w_p\|_{L^{\infty}(\mathcal{G})})$,	we denote by
	\[
	m(t):=|\{w_p>t\}|
	\]
	the distribution function of $w_p$, and by $N(t)$ the number of elements
	in the level set $\{w_p=t\}$. Observe first that for every
	$t\in(0,\|w_p\|_{L^{\infty}(\mathcal{G})})$ the superlevel set $\{w_p>t\}$ is a nonempty,
	open, proper subset of the connected graph $\mathcal G$; its boundary is
	therefore nonempty and contained in $\{w_p=t\}$, whence
	\begin{equation}
		\label{eq:Nge1}
		N(t)\ge1
		\qquad
		\text{for every }t\in(0,\|w_p\|_{L^{\infty}(\mathcal{G})}).
	\end{equation}
	The first step is a lower bound for the Dirichlet energy in terms of
	the distribution function. For every regular level $t$, the level set
	$w_p^{-1}(t)$ consists of finitely many points
	\(
	x_1(t),\ldots,x_{N(t)}(t),
	\)
	all lying in the interior of edges and satisfying
	$w_p'(x_j(t))\neq0$: indeed, on each of the finitely many maximal
	monotonicity intervals of $w_p$, the equation $w_p=t$ has at most one
	solution. Introducing
	\[
	S(t):=
	\sum_{j=1}^{N(t)}
	|w_p'(x_j(t))|^{p-1},
	\]
	the coarea formula on metric graphs  (see~\cite[p.~22]{AST17})
	yields
	\begin{equation}
		\label{eq:coarea_Lp}
		\int_{\mathcal G}|w_p'|^p\,dx
		=
		\int_0^{\|w_p\|_{L^{\infty}(\mathcal{G})}}
		S(t)\,dt.
	\end{equation}
	The same piecewise monotone structure shows that $m$ is absolutely
	continuous and strictly decreasing on $(0,\|w_p\|_{L^{\infty}(\mathcal{G})})$, with
	\begin{equation}
		\label{eq:mdiff}
		-m'(t)
		=
		\sum_{j=1}^{N(t)}
		\frac1{|w_p'(x_j(t))|}
		\qquad
		\text{for a.e. }t.
	\end{equation}
	Applying H\"older's inequality with conjugate exponents
	$p$ and $p/(p-1)$ gives
	\[
	N(t)
	=
	\sum_{j=1}^{N(t)}
	|w_p'(x_j(t))|^{\frac{p-1}{p}}
	|w_p'(x_j(t))|^{-\frac{p-1}{p}}
	\le
	S(t)^{1/p}
	(-m'(t))^{(p-1)/p},
	\]
	or equivalently,
	\begin{equation}
		\label{eq:Holder_level}
		S(t)
		\ge
		\frac{N(t)^p}
		{(-m'(t))^{p-1}}.
	\end{equation}
	Combining \eqref{eq:coarea_Lp}, \eqref{eq:Holder_level}, and
	\eqref{eq:Nge1}, we obtain
	\begin{equation}
		\label{eq:coarea}
		\int_{\mathcal G}|w_p'|^p\,dx
		\ge
		\int_0^{\|w_p\|_{L^{\infty}(\mathcal{G})}}
		\frac{N(t)^p}
		{(-m'(t))^{p-1}}
		\,dt
		\ge
		\int_0^{\|w_p\|_{L^{\infty}(\mathcal{G})}}
		(-m'(t))^{1-p}\,dt.
	\end{equation}
	The energy estimate obtained above naturally leads to the decreasing
	rearrangement of $w_p$. Let
	\(
	w_p^*:[0,L]\longrightarrow[0,\infty)
	\)
	be the decreasing rearrangement of $w_p$, defined by
	\[
	w_p^*(s)
	:=
	\inf\{t\ge0:\,m(t)<s\}.
	\]
	Since $m$ is continuous and strictly decreasing, with $m(t)\to L$ as
	$t\to0$ (recall that $w_p>0$ outside the finite set
	$\mathcal V^D$) and $m(t)\to0$ as $t\to\|w_p\|_{L^{\infty}(\mathcal{G})}$, the
	rearrangement $w_p^*$ is nothing but the inverse function of $m$,
	extended by $w_p^*(0)=\|w_p\|_{L^{\infty}(\mathcal{G})}$. In particular, $w_p$ and
	$w_p^*$ are equimeasurable, so that
	\[
	\int_0^L w_p^*(s)\,ds
	=
	\int_{\mathcal G}w_p(x)\,dx.
	\]
	Let us now compare the Dirichlet energies of $w_p$ and $w_p^*$. Since
	\(
	w_p^*(m(t))=t
\text{ for every }t\in(0,\|w_p\|_{L^{\infty}(\mathcal{G})}),
	\)
	differentiating this identity at almost every $t$ yields
	\(
	(w_p^*)'(m(t))\,m'(t)=1,
\text{ hence }
	|(w_p^*)'(m(t))|^p
	=
	(-m'(t))^{-p},
	\)
	and the change of variables $s=m(t)$  gives
	\begin{equation}
		\label{eq:rear_id}
		\int_0^L |(w_p^*)'(s)|^p\,ds
		=
		\int_0^{\|w_p\|_{L^{\infty}(\mathcal{G})}}
		(-m'(t))^{1-p}\,dt.
	\end{equation}
	Combining \eqref{eq:coarea} with \eqref{eq:rear_id}, we arrive at the
	rearrangement inequality
	\begin{equation}
		\label{eq:rear_ineq}
		\int_{\mathcal G}|w_p'|^p\,dx
		\ge
		\int_0^L |(w_p^*)'(s)|^p\,ds:
	\end{equation}
	passing to the rearrangement can only decrease the $p$-Dirichlet
	energy.
	
	To apply the variational characterization of the torsional rigidity on
	the interval, it remains to verify that the rearrangement satisfies
	the appropriate boundary condition. Let $v_0\in\mathcal{V}^D$. Since
	$w_p(v_0)=0$ and $w_p$ is continuous on $\mathcal G$, every
	neighborhood of $v_0$ contains points where $w_p<t$ for each $t>0$.
	Consequently,
	\(
	m(t)<L
\text{ for every }t>0,
	\)
	and therefore
	\[
	w_p^*(L)
	=
	\inf\{t\ge0:\,m(t)<L\}
	=
	0.
	\]
	Thus $w_p^*\in W^{1,p}_{\{L\}}(J)$ .
 Combining the equimeasurability
	identity with \eqref{eq:rear_ineq}, we obtain
	\[
	\frac{\left(\int_{\mathcal G}w_p\,dx\right)^p}
	{\int_{\mathcal G}|w_p'|^p\,dx}
	\le
	\frac{\left(\int_0^L w_p^*\,ds\right)^p}
	{\int_0^L |(w_p^*)'|^p\,ds},
	\]
	and Proposition~\ref{prop:char}, applied on $J$ and $\mathcal{G}$ gives
	\[
	\frac{\left(\int_0^L w_p^*\,ds\right)^p}
	{\int_0^L |(w_p^*)'|^p\,ds}
	\le
	T_p(J,\{L\}),
	\]
	and therefore
	\(
	T_p(\mathcal G,\mathcal{V}^D)
	\le
	T_p(J,\{L\}),
	\)
	which proves \eqref{eq:SV}.
	
	We now discuss the equality case. Suppose first that equality holds in
	\eqref{eq:SV}. Then equality must hold throughout the preceding chain
	of inequalities; in particular, equality holds in \eqref{eq:coarea}.
	The only loss in passing from
	\[
	\int_0^{\|w_p\|_{L^{\infty}(\mathcal{G})}}
	\frac{N(t)^p}{(-m'(t))^{p-1}}\,dt
	\qquad\text{to}\qquad
	\int_0^{\|w_p\|_{L^{\infty}(\mathcal{G})}}
	(-m'(t))^{1-p}\,dt
	\]
	comes from the estimate $N(t)\ge1$, so equality forces
	\(
	N(t)=1
 \)  for almost every regular level   \(t\in(0,\|w_p\|_{L^{\infty}(\mathcal{G})}).
	\)
	Thus almost every nonempty superlevel set of $w_p$ has a single
	boundary point. This is a strong topological constraint: the same
	argument as in the proof of \cite[Lemma~4.3]{BKKM17} shows that it
	excludes both branching vertices and cycles. Hence $\mathcal G$ is necessarily a path graph.
	
	It remains to determine the Dirichlet set. Since $w_p$ is positive on
	$\mathcal G\setminus\mathcal V^D$, every sufficiently small positive level
	set of $w_p$ has one boundary point near each Dirichlet endpoint.
	Hence, if $\mathcal G$ had two Dirichlet endpoints, then
	$N(t)\ge2$ for all sufficiently small $t>0$, contradicting the equality
	condition $N(t)=1$ for almost every regular level. Therefore
	$\mathcal G$ has exactly one Dirichlet endpoint. Since
	$|\mathcal G|=L$, it follows that, up to subdivision of vertices of
	degree two, $\mathcal G$ is an interval of length $L$ with a single
	Dirichlet endpoint.
	Conversely, if $\mathcal G$ is an interval of length $L$ with a single
	Dirichlet endpoint, then Example~\ref{ex:J0} gives
	\[
	T_p(\mathcal G,\mathcal{V}^D)
	=
	T_p(J,\{L\})
	=
	\Bigl(\frac{L^{q+1}}{q+1}\Bigr)^{p-1}.
	\]
	Equality is therefore attained precisely by the Dirichlet--Neumann
	interval, and the proof is complete.
\end{proof}
	\section{Spectral estimates}
	\label{sec:spectral}
	
	This section relates the $p$-torsional rigidity to the first eigenvalue of
	the Dirichlet--Kirchhoff $p$-Laplacian, which is given by the Rayleigh
	quotient
	\begin{equation}
		\label{eq:lambda1p}
		\lambda_{1,p}(\mathcal G,\mathcal{V}^D)
		:=
		\inf_{\substack{u\in W^{1,p}_{\mathcal{V}^D}(\mathcal G)\\ u\neq 0}}
		\frac{\int_{\mathcal G}|u'|^p\,dx}
		{\int_{\mathcal G}|u|^p\,dx}.
	\end{equation}
	By \cite[Theorems~3.1 and~3.3]{DR}, the infimum in \eqref{eq:lambda1p} is
	attained, the first eigenvalue is simple, and the associated eigenfunction
	may be chosen strictly positive on $\mathcal G\setminus \mathcal{V}^D$.
	
	\subsection{The \texorpdfstring{$p$}{p}-P\'olya--Szeg\H{o} inequality}
	
	The classical P\'olya--Szeg\H{o} inequality for $p=2$ states that
	$\lambda_{1,2}T_2<|\mathcal G|$ (see \cite[Proposition 5.1]{MuPl}). For $p\neq 2$, the dimensionally homogeneous
	form is obtained by raising the product $\lambda_{1,p}T_p$ to the power
	$1/(p-1)$.
	
	\begin{proposition}
		\label{prop:PS}
		Let $\mathcal G$ be a connected compact metric graph, let
		$\mathcal{V}^D\subseteq \mathcal{V}$ be a nonempty set of degree-one vertices, and let
		$1<p<\infty$. Then
		\begin{equation}
			\label{eq:PS}
			\lambda_{1,p}(\mathcal G,\mathcal{V}^D)T_p(\mathcal G,\mathcal{V}^D)
			<
			|\mathcal G|^{p-1}.
		\end{equation}
	\end{proposition}
	
	\begin{proof}
		Let $w_p\in W^{1,p}_{\mathcal{V}^D}(\mathcal G)$ be the $p$-torsion function. By \eqref{eq:torsion_def} and \eqref{eq:torsion_energy}, we have
		\[
		\int_{\mathcal G}|w_p'|^p\,dx
		=
		\int_{\mathcal G}w_p\,dx
		=
		T_p(\mathcal G,\mathcal{V}^D)^{1/(p-1)}.
		\]
		Moreover, H\"older's inequality gives
		\[
		\int_{\mathcal G}w_p\,dx
		\le
		\left(\int_{\mathcal G}w_p^p\,dx\right)^{1/p}
		|\mathcal G|^{1/q},
		\]
		and hence
		\[
		\int_{\mathcal G}w_p^p\,dx
		\ge
		\frac{T_p(\mathcal G,\mathcal{V}^D)^{p/(p-1)}}{|\mathcal G|^{p-1}}.
		\]
		Using $w_p$ as a test function in the Rayleigh quotient \eqref{eq:lambda1p} therefore yields
		\[
		\lambda_{1,p}(\mathcal G,\mathcal{V}^D)
		\le
		\frac{\int_{\mathcal G}|w_p'|^p\,dx}
		{\int_{\mathcal G}w_p^p\,dx}
		\le
		\frac{|\mathcal G|^{p-1}}{T_p(\mathcal G,\mathcal{V}^D)},
		\]
which shows \eqref{eq:PS}
		It remains to exclude equality. If equality held, then equality would also
		hold in H\"older's inequality, and $w_p$ would be constant almost
		everywhere on $\mathcal G$; being continuous, it would then be constant
		on $\mathcal G$. This is impossible because $w_p=0$ on $\mathcal{V}^D$, while
		$w_p>0$ on $\mathcal G\setminus \mathcal{V}^D$. The inequality is therefore
		strict.
	\end{proof}
	
	\subsection{The Kohler--Jobin inequality}
	
	We now turn to the second main isoperimetric inequality of the paper,
	namely a sharp Kohler--Jobin inequality relating the $p$-torsional
	rigidity to the first Dirichlet eigenvalue of the $p$-Laplacian.
	The quantity we are interested in is
	\[
	\lambda_{1,p}(\mathcal G,\mathcal V^D)\,
	T_p(\mathcal G,\mathcal V^D)^{\frac{p}{2p-1}},
	\]
	which is invariant under a uniform scaling of the edge lengths. Hence,
	as in the classical Euclidean setting, the optimal constant is determined
	by the corresponding one-dimensional model problem.
	
	Our proof follows the rearrangement approach studied in \cite{Brasco} for the
	$p$-Laplacian and adapted to compact metric graphs in \cite{MuPl}
	in the linear case. Starting from the positive first eigenfunction, we construct a suitable comparison
	interval and a rearranged function, which allow us to compare both the first
	eigenvalue and an auxiliary torsional quantity with those of the original
	graph.
	Accordingly, let
	\(
	J=[0,1],
	\)
	equipped with the Dirichlet condition at the endpoint $1$. The first
	Dirichlet eigenvalue of the $p$-Laplacian on $J$ is given by
	\[
	\lambda_{1,p}(J,\{1\})
	=
	(p-1)\Bigl(\frac{\pi_p}{2}\Bigr)^p,
	\]
	where
	\(
	\pi_p=\frac{2\pi}{p\sin(\pi/p)}
	\)
	denotes the first positive zero of the generalized sine function
	$\sin_p$ introduced by Lindqvist \cite{Lindqvist}; see also
	\cite{LE11}. Moreover, Example~\ref{ex:J0} yields
	\[
	T_p(J,\{1\})
	=
	\Bigl(\frac{1}{q+1}\Bigr)^{p-1}.
	\]
	The following theorem shows that the value of the scale-invariant product on
	the model interval is in fact optimal.

	\begin{theorem}
		\label{thm:KJ}
		Let $\mathcal G$ be a connected compact metric graph,
		let $\mathcal V^D\subseteq\mathcal V$ be a nonempty set of degree-one
		vertices, and let $1<p<\infty$. Then
		\begin{equation}
			\label{eq:KJ}
			\lambda_{1,p}(\mathcal G,\mathcal V^D)\,
			T_p(\mathcal G,\mathcal V^D)^{\frac{p}{2p-1}}
			\ge
			C_p^{\mathrm{KJ}},
		\end{equation}
		where 
		\begin{equation}	\label{def:KJconst}
			C_p^{\mathrm{KJ}}
			:=
			\lambda_{1,p}(J,\{1\})
			\,T_p(J,\{1\})^{\frac{p}{2p-1}}
			=
			(p-1)
			\Bigl(\frac{\pi_p}{2}\Bigr)^p
			\Bigl(\frac{1}{q+1}\Bigr)^{\frac{p(p-1)}{2p-1}}.
		\end{equation}
		Equality holds if and only if $\mathcal G$ is an interval with a single Dirichlet endpoint.
	\end{theorem}
	
	For $p=2$, the constant reduces to
	\(
	C_2^{\mathrm{KJ}}
	=
	\Bigl(\frac{\pi}{\sqrt[3]{24}}\Bigr)^2,
	\)
	recovering the sharp Kohler--Jobin inequality of
	\cite[Theorem~5.8]{MuPl}.
	
	The proof is based on the level-set structure of the positive first
	eigenfunction. We therefore introduce the distribution function of its
	superlevel sets together with the corresponding nonlinear flux across the
	level sets.
	
	Let $\psi\in\WpD$ be the positive first eigenfunction normalized by
	\(
	\|\psi\|_{L^p(\mathcal G)}=1,
	\)
	and set
	\(
	M:=\|\psi\|_{L^\infty(\mathcal G)}.
	\)
	For $t\in(0,M)$, let
	\[
	\alpha(t)
	:=
	|\{x\in\mathcal G:\psi(x)>t\}|
	\]
	denote the distribution function of the superlevel set
	$\{\psi>t\}$, and let
	\[
	\gamma_p(t)
	:=
	\sum_{x\in\psi^{-1}(t)}
	|\psi'(x)|^{p-1}
	\]
	denote the total nonlinear flux across the level set
	$\{\psi=t\}$.
	We introduce
	\[
	I(t)
	:=
	\int_t^M
	\frac{\alpha(s)^q}{\gamma_p(s)^{\,q-1}}
	\,ds,
	\]
	whose finiteness will be proved in the next lemma.
	Accordingly, we define the associated modified $p$-torsional rigidity by
	\[
	T_{p,\mathrm{mod}}(\mathcal{G})
	:=
	I(0)^{p-1}.
	\]
	\begin{lemma}
		\label{lem:modtorsion}
 $T_{p,\mathrm{mod}}(\mathcal{G})$ is well defined and satisfies
		\(
		T_{p,\mathrm{mod}}(\mathcal{G})
		\le
		T_p(\mathcal G,\mathcal V^D).
		\)
	\end{lemma}
	\begin{proof}
		We first verify that the auxiliary
		quantities $\gamma_p$, $I$, and $T_{p,\mathrm{mod}}(\mathcal G)$ are
		well defined; we then construct, out of the level-set structure of
		$\psi$, an explicit test function $\eta\in\WpD$ whose P\'olya quotient
		equals $I(0)^{p-1}$. The claimed inequality then follows from the
		variational characterisation of Proposition~\ref{prop:char}.

		For almost every $t\in(0,M)$ the level
		set $\{\psi=t\}$ consists of finitely many points lying in the interior
		of edges, at each of which $\psi'\neq0$; moreover, $\alpha$ is
		absolutely continuous with
		\[
		-\alpha'(t)
		=
		\sum_{x\in\psi^{-1}(t)}
		\frac1{|\psi'(x)|}
		\qquad\text{for a.e. }t\in(0,M),
		\]
		and $\gamma_p(t)\in(0,\infty)$ for almost every $t$. Since, for
		$t\in(0,M)$, the superlevel set $\{\psi>t\}$ is nonempty and a proper
		subset of $\mathcal G$, its boundary is nonempty; hence the number
		$N(t)$ of points in the level set $\{\psi=t\}$ satisfies $N(t)\ge1$.
		Applying H\"older's inequality with conjugate exponents $p$ and
		$p/(p-1)$, exactly as in \eqref{eq:Holder_level}, we obtain
		\begin{equation}
			\label{eq:levelset_holder}
			1
			\le
			N(t)^p
			\le
			\gamma_p(t)\bigl(-\alpha'(t)\bigr)^{p-1}
			\qquad
			\text{for a.e. }t\in(0,M).
		\end{equation}
		Introducing
		\(
		A(t):=
		\frac{\alpha(t)^q}{\gamma_p(t)^{q-1}},
		\)
		and observing that \eqref{eq:levelset_holder} is equivalent, by
		$(p-1)(q-1)=1$, to $\gamma_p(t)^{1-q}\le-\alpha'(t)$, we deduce
		\begin{equation}
			\label{eq:A_estimate}
			A(t)
			\le
			\alpha(t)^q\bigl(-\alpha'(t)\bigr)
			\qquad
			\text{for a.e. }t\in(0,M).
		\end{equation}
		Consequently, using the absolute continuity of $\alpha$ together with
		$\alpha(M)=0$,
		\begin{equation}\label{eq:est_on_I}
			I(t)
			=
			\int_t^M A(s)\,ds
			\le
			\int_t^M
			\alpha(s)^q\bigl(-\alpha'(s)\bigr)\,ds
			=
			\frac{\alpha(t)^{q+1}}{q+1}
			<\infty,
		\end{equation}
		showing that both $I$ and $T_{p,\mathrm{mod}}(\mathcal{G})$ are well
		defined.

		We define
		\[
		F(t)
		:=
		\int_0^t
		\left(
		\frac{\alpha(s)}{\gamma_p(s)}
		\right)^{q-1}
		\,ds,
		\qquad
		0\le t\le M,
		\]
		and set
		\(
		\eta(x):=F(\psi(x)),
		\) \(x\in\mathcal G.
		\)
		The function $F$ is well defined and finite: since
		$(\alpha/\gamma_p)^{q-1}=A/\alpha$, the estimate \eqref{eq:A_estimate}
		yields
		\[
		\left(\frac{\alpha(s)}{\gamma_p(s)}\right)^{q-1}
		\le
		\alpha(s)^{q-1}\bigl(-\alpha'(s)\bigr)
		\qquad\text{for a.e. }s\in(0,M),
		\]
		and the right-hand side is integrable, with
		\[
		\int_0^M\alpha(s)^{q-1}\bigl(-\alpha'(s)\bigr)\,ds
		=\frac{\alpha(0)^q}{q}
		=\frac{|\mathcal G|^q}{q}.
		\]
		Being the indefinite
		integral of a nonnegative $L^1$ function, $F$ is absolutely continuous
		and non-decreasing on $[0,M]$ with $F(0)=0$.
		We claim that $\eta\in\WpD$.
		Let us compute $\int_{\mathcal G}\eta\,dx$. Applying the coarea
		formula to the composition $\eta=F(\psi)$ gives
		\[
		\int_{\mathcal G}\eta\,dx
		=
		\int_0^M
		F(t)\bigl(-\alpha'(t)\bigr)\,dt.
		\]
		Since $F(0)=0$ and $\alpha(M)=0$, and both $F$ and $\alpha$ are
		absolutely continuous, integration by parts together with the identity
		\[
		F'(t)\alpha(t)
		=
		\frac{\alpha(t)^q}{\gamma_p(t)^{q-1}}
		=
		A(t)
		\]
		yields
		\[
		\int_{\mathcal G}\eta\,dx
		=
		\int_0^M F'(t)\alpha(t)\,dt
		=
		\int_0^M
		A(t)\,dt
		=
		I(0).
		\]
		Next we compute the Dirichlet energy of $\eta$. By the chain rule,
		\(
		|\eta'|^p
		=
		|F'(\psi)|^p|\psi'|^p,
		\)
		and the coarea formula, applied as in \eqref{eq:coarea_Lp} with the
		weight $|\psi'|^p=|\psi'|^{p-1}\cdot|\psi'|$, gives
		\[
		\int_{\mathcal G}|\eta'|^p\,dx
		=
		\int_0^M
		|F'(t)|^p\gamma_p(t)\,dt.
		\]
		Since
		\(
		p(q-1)=q,
		\)
		we have
		\[
		|F'(t)|^p\gamma_p(t)
		=
		\left(
		\frac{\alpha(t)}{\gamma_p(t)}
		\right)^q
		\gamma_p(t)
		=
		A(t),
		\]
		and hence
		\[
		\int_{\mathcal G}|\eta'|^p\,dx
		=
		\int_0^M
		A(t)\,dt
		=
		I(0)
		<\infty.
		\]
		In particular, $\eta'\in L^p(\mathcal G)$, and
		$\eta\not\equiv0$ since $\int_{\mathcal G}\eta\,dx=I(0)>0$. Therefore
		$\eta$ is an admissible test function in the variational
		characterisation of the $p$-torsional rigidity
		(Proposition~\ref{prop:char}), which yields
		\[
		T_p(\mathcal G,\mathcal V^D)
		\ge
		\frac{\left(\int_{\mathcal G}\eta\,dx\right)^p}
		{\int_{\mathcal G}|\eta'|^p\,dx}
		=
		\frac{I(0)^p}{I(0)}
		=
		I(0)^{p-1}
		=
		T_{p,\mathrm{mod}}(\mathcal{G}).\qedhere
		\]
	\end{proof}
	The next step is to construct a one-dimensional comparison function by
	means of a rearrangement argument. In the Euclidean setting, this approach
	was studied by Brasco for the $p$-Laplacian \cite{Brasco}, while
	Mugnolo and Pl\"umer adapted it to compact metric graphs in the linear case
	\cite[Lemma~5.17]{MuPl}. The following lemma extends the graph
	rearrangement argument to the nonlinear setting.
	
	Motivated by the definition of $T_{p,\mathrm{mod}}(\mathcal{G})$, we introduce the
	interval
	\(
	J_\psi:=(0,\ell_\psi),\) where
	\(\ell_\psi:=\bigl((q+1)I(0)\bigr)^{1/(q+1)}.
	\)
	With this choice, we have
	\begin{equation}\label{eq:Modtor}
	T_p(J_\psi,\{\ell_\psi\})
=
\left(
\frac{\ell_\psi^{q+1}}{q+1}
\right)^{p-1}
=
I(0)^{p-1}
=
T_{p,\mathrm{mod}}(\mathcal{G}).
	\end{equation}

	\begin{lemma}
		\label{lem:KJrearrange}
		There exists a rearranged function
		\(
		\psi^*\in W^{1,p}(J_\psi),\) \(	\psi^*(\ell_\psi)=0,
		\)
		such that
		\begin{equation}
			\label{eq:KJrear1}
			\int_{J_\psi}|(\psi^*)'|^p\,dx
			=
			\int_{\mathcal G}|\psi'|^p\,dx,
		\end{equation}
		and
		\begin{equation}
			\label{eq:KJrear2}
			\int_{J_\psi}|\psi^*|^p\,dx
			\ge
			\int_{\mathcal G}|\psi|^p\,dx.
		\end{equation}
	\end{lemma}
	\begin{proof}
	The construction is based on the idea of \emph{transplantation with equal
		Dirichlet integrals}, which lies at the heart of the
	Kohler--Jobin method. We construct a function $\psi^*$ on the calibrated
	interval $J_\psi$ whose level structure is determined by that of $\psi$,
	so that the Dirichlet energy is preserved exactly, while the comparison
	of superlevel sets ensures that no $L^p$-mass is lost. Rather than using
	the level parameter $t$, it is more convenient to work with the modified
	torsion variable $\tau=I(t)$. We therefore let
	\[
	\varphi:=I^{-1}:[0,I(0)]\to[0,M]
	\]
	be the inverse of $I$. Since $I$ is absolutely continuous, being the
	indefinite integral of $A\in L^1(0,M)$, and is strictly decreasing with
	$I'=-A\neq0$ almost everywhere, its inverse $\varphi$ is continuous,
	strictly decreasing and absolutely
	continuous, with $\varphi'\neq0$ almost everywhere. We also define
	\[
	R(\tau):=((q+1)\tau)^{1/(q+1)}.
	\]
	Then $R$ maps $[0,I(0)]$ increasingly onto $[0,\ell_\psi]$. It serves as
	the model profile since, on the Dirichlet--Neumann interval, the quantity
	corresponding to $\alpha(\varphi(\tau))$ is exactly $R(\tau)$, as shown
	in Example~\ref{ex:J0}. The next estimate shows that the superlevel sets
	of $\psi$ on $\mathcal G$ are always at least as large as those of this
	model.
	
	Evaluating \eqref{eq:est_on_I} at $t=\varphi(\tau)$ and then inverting the
	resulting inequality gives the desired comparison between the
	distribution function of $\psi$ and the model profile:
	\begin{equation}
		\label{eq:alpha_R_estimate}
		\alpha(\varphi(\tau))
		\ge
		R(\tau)
		\qquad
		\text{for }0\le \tau\le I(0).
	\end{equation}
	We are now ready to transplant $\psi$ onto $J_\psi$. The idea is to
	prescribe the derivative of the transplanted profile so that the
	Dirichlet energy is reproduced exactly, level by level. To this end, we
	define a function $\kappa$ on $[0,I(0)]$ by requiring it to satisfy
	\[
	-R(\tau)\kappa'(\tau)
	=
	-\alpha(\varphi(\tau))\varphi'(\tau),
	\qquad
	\kappa(I(0))=0,
	\]
	or equivalently,
	\[
	\kappa(\tau)
	=
	\int_\tau^{I(0)}
	\frac{\alpha(\varphi(s))}{R(s)}
	\bigl(-\varphi'(s)\bigr)\,ds.
	\]
	Observe that $\varphi$ is decreasing, so the integrand is nonnegative.
	Consequently, $\kappa$ is nonnegative and non-increasing. Moreover,
	\eqref{eq:alpha_R_estimate} gives
	$\alpha(\varphi(s))/R(s)\ge1$, and comparison with the identity
	\[
	\varphi(\tau)
	=
	\int_\tau^{I(0)}
	\bigl(-\varphi'(s)\bigr)\,ds
	\]
	shows that
	\[
	\kappa(\tau)\ge\varphi(\tau)
	\qquad
	\text{for }0\le\tau\le I(0).
	\]
	Thus the transplanted profile dominates the original level
	parametrisation pointwise, and this comparison will be responsible for
	the $L^p$-estimate below.
	
	We now define $\psi^*$ on $J_\psi$ by
	\[
	\psi^*(R(\tau))
	:=
	\kappa(\tau),
	\qquad
	0\le\tau\le I(0),
	\]
	which is well defined since $R$ is a bijection from $[0,I(0)]$ onto
	$[0,\ell_\psi]$. Because $R(I(0))=\ell_\psi$ and
	$\kappa(I(0))=0$, the function $\psi^*$ satisfies the required
	Dirichlet condition
	\(
	\psi^*(\ell_\psi)=0.
	\)
	
	We next prove \eqref{eq:KJrear1}. Differentiating
	$R(\tau)=((q+1)\tau)^{1/(q+1)}$ gives
	\[
	R'(\tau)=R(\tau)^{-q},
	\]
	and using the change of variables $x=R(\tau)$ together with the defining
	equation for $\kappa$, we obtain
	\[
	\int_{J_\psi}|(\psi^*)'|^p\,dx
	=
	\int_0^{I(0)}
	\frac{|\kappa'(\tau)|^p}{R'(\tau)^{p-1}}\,d\tau
	=
	\int_0^{I(0)}
	\alpha(\varphi(\tau))^p
	|\varphi'(\tau)|^p\,d\tau.
	\]
		It remains to recognise the right-hand side as the Dirichlet energy of
		$\psi$. As \(\varphi=I^{-1}\) and
		\(
		I'(t)
		=
		-\frac{\alpha(t)^q}{\gamma_p(t)^{q-1}},
		\)
		we have
		\(
		|\varphi'(\tau)|
		=
		\frac{\gamma_p(\varphi(\tau))^{q-1}}
		{\alpha(\varphi(\tau))^q}.
		\)
		Since \(p(q-1)=q\), it follows that
		\[
		\alpha(\varphi(\tau))^p|\varphi'(\tau)|^p
		=
		\frac{\gamma_p(\varphi(\tau))^q}
		{\alpha(\varphi(\tau))^q}.
		\]
		Therefore, changing variables \(t=\varphi(\tau)\), so that
		$d\tau=|I'(t)|\,dt=\bigl(\alpha(t)^q/\gamma_p(t)^{q-1}\bigr)\,dt$,
		\[
		\int_{J_\psi}|(\psi^*)'|^p\,dx
		=
		\int_0^M \gamma_p(t)\,dt
		=
		\int_{\mathcal G}|\psi'|^p\,dx,
		\]
		where the last equality follows from the coarea formula, as in
		\eqref{eq:coarea_Lp}.
		
	We now turn to the proof of \eqref{eq:KJrear2}. Since $\psi^*$ is
	non-increasing and satisfies $\psi^*(\ell_\psi)=0$, the layer-cake
	formula, together with the change of variables $x=R(\tau)$ and the
	identity $R'(\tau)R(\tau)^q=1$, yields
	\[
	\int_{J_\psi}|\psi^*|^p\,dx
	=
	p\int_0^{I(0)}
	\kappa(\tau)^{p-1}
	\alpha(\varphi(\tau))
	|\varphi'(\tau)|\,d\tau.
	\]
	We now use the pointwise estimate
	$\kappa\ge\varphi$. Since the weight
	$\alpha(\varphi(\tau))|\varphi'(\tau)|$ is nonnegative, it follows that
	\[
	\int_{J_\psi}|\psi^*|^p\,dx
	\ge
	p\int_0^{I(0)}
	\varphi(\tau)^{p-1}
	\alpha(\varphi(\tau))
	|\varphi'(\tau)|\,d\tau
	=
	\int_{\mathcal G}|\psi|^p\,dx,
	\]
	where the last equality is simply the layer-cake representation for
	$\psi$, after the change of variables $t=\varphi(\tau)$. This establishes
	\eqref{eq:KJrear2} and completes the proof.
	\end{proof}
	Before proving Theorem~\ref{thm:KJ}, we show a simple consequence of the
	scaling laws for $\lambda_{1,p}$ and $T_p$.
	\begin{lemma}
		\label{lem:KJscaling}
		Let $(\mathcal G,\mathcal V^D)$ be a compact metric graph with a nonempty
		Dirichlet set, and let $c>0$. Denote by
		$c\mathcal G$ the graph obtained by multiplying all edge lengths by $c$.
		Then
		\[
		\lambda_{1,p}(c\mathcal G,\mathcal V^D)\,
		T_p(c\mathcal G,\mathcal V^D)^{\frac{p}{2p-1}}
		=
		\lambda_{1,p}(\mathcal G,\mathcal V^D)\,
		T_p(\mathcal G,\mathcal V^D)^{\frac{p}{2p-1}}.
		\]
	\end{lemma}
	
	\begin{proof}
	The scaling law for $T_p$ was proved in
	Lemma~\ref{lem:scaling}, and the same change-of-variables argument yields
	the corresponding scaling law for $\lambda_{1,p}$. Consequently,
		\[
		\lambda_{1,p}(c\mathcal G,\mathcal V^D)
		T_p(c\mathcal G,\mathcal V^D)^{\frac{p}{2p-1}}
		=
		c^{-p}
		c^{(2p-1)\frac{p}{2p-1}}
		\lambda_{1,p}(\mathcal G,\mathcal V^D)
		T_p(\mathcal G,\mathcal V^D)^{\frac{p}{2p-1}},
		\]
		and the exponent of $c$ vanishes since
		\(
		-p+(2p-1)\frac{p}{2p-1}=0.
		\)
	\end{proof}
	\begin{proof}[Proof of Theorem~\ref{thm:KJ}]
	We compare $\mathcal G$ with the interval $J_\psi$ constructed from the
	first eigenfunction $\psi$. The definition of $\ell_\psi$ through the
	modified $p$-torsional quantity is precisely what makes this comparison work:
	it yields simultaneous bounds for both the first eigenvalue and the
	torsional rigidity, whereas these quantities would scale in opposite
	directions for an interval of arbitrary length. Since the
	Kohler--Jobin product is scale invariant and therefore equals
	$C_p^{\mathrm{KJ}}$ on every Dirichlet--Neumann interval, the theorem
	follows immediately.
		
		Let $\psi\in\WpD$ be the positive first eigenfunction normalized by
		\(
		\|\psi\|_{L^p(\mathcal G)}=1.
		\)
		We use the notation introduced above and set
		\(
		J_\psi=(0,\ell_\psi),\) where \(
		\ell_\psi=\bigl((q+1)I(0)\bigr)^{1/(q+1)}.
		\)
		By Lemma~\ref{lem:KJrearrange}, there exists
		$\psi^*\in W^{1,p}(J_\psi)$ with $\psi^*(\ell_\psi)=0$ such that
		\[
		\int_{J_\psi}|(\psi^*)'|^p\,dx
		=
		\int_{\mathcal G}|\psi'|^p\,dx,
		\qquad
		\int_{J_\psi}|\psi^*|^p\,dx
		\ge
		\int_{\mathcal G}|\psi|^p\,dx.
		\]
		Since $\psi$ is normalized and realizes the first eigenvalue, we have
		\[
		\int_{\mathcal G}|\psi|^p\,dx=1,
		\qquad
		\int_{\mathcal G}|\psi'|^p\,dx
		=
		\lambda_{1,p}(\mathcal G,\mathcal V^D).
		\]
		In other words, the transplantation does not increase the Rayleigh
		quotient: $\psi^*$ is an admissible test function for the Rayleigh
		quotient on the interval $J_\psi$ with Dirichlet condition at
		$\ell_\psi$, and therefore
		\begin{equation}
			\label{eq:KJ_lambda_comparison}
			\lambda_{1,p}(J_\psi,\{\ell_\psi\})
			\le
			\frac{\int_{J_\psi}|(\psi^*)'|^p\,dx}
			{\int_{J_\psi}|\psi^*|^p\,dx}
			\le
			\lambda_{1,p}(\mathcal G,\mathcal V^D).
		\end{equation}
		Combining \eqref{eq:Modtor} with Lemma~\ref{lem:modtorsion}, we conclude
		that the comparison interval also has smaller torsional rigidity:
		\begin{equation}
			\label{eq:KJ_torsion_comparison}
			T_p(J_\psi,\{\ell_\psi\})
			\le
			T_p(\mathcal G,\mathcal V^D).
		\end{equation}
		It remains to identify the value of the Kohler--Jobin product on the
		interval $J_\psi$. By Lemma \ref{lem:KJscaling}, the
		quantity
		\(
		\lambda_{1,p}(J,\mathcal V^D)\,
		T_p(J,\mathcal V^D)^{\frac{p}{2p-1}}
		\)
		is invariant under changes of the length of the interval. Hence the
		Dirichlet--Neumann interval $(J_\psi,\{\ell_\psi\})$ has the same
		Kohler--Jobin product as the unit model interval used in
		\eqref{def:KJconst}, that is,
		\begin{equation}
			\label{eq:KJ_interval_value}
			\lambda_{1,p}(J_\psi,\{\ell_\psi\})\,
			T_p(J_\psi,\{\ell_\psi\})^{\frac{p}{2p-1}}
			=
			C_p^{\mathrm{KJ}}.
		\end{equation}
		Combining \eqref{eq:KJ_interval_value} with
		\eqref{eq:KJ_lambda_comparison} and \eqref{eq:KJ_torsion_comparison},
		and using that the map $s\mapsto s^{p/(2p-1)}$ is increasing on
		$(0,\infty)$, we obtain
		\[
		C_p^{\mathrm{KJ}}
		\le
		\lambda_{1,p}(\mathcal G,\mathcal V^D)\,
		T_p(\mathcal G,\mathcal V^D)^{\frac{p}{2p-1}}.
		\]
Finally we identify the equality cases. If $\mathcal G$ is an interval
	with a single Dirichlet endpoint, then equality in \eqref{eq:KJ} follows
	immediately from the scale invariance of the Kohler--Jobin product together
	with \eqref{def:KJconst}. Conversely, suppose that equality holds in
	\eqref{eq:KJ}. Then, by
	\eqref{eq:KJ_interval_value}, \eqref{eq:KJ_lambda_comparison}, and
	\eqref{eq:KJ_torsion_comparison}, every inequality in
\begin{align*}
	C_p^{\mathrm{KJ}}
	&=
	\lambda_{1,p}(J_\psi,\{\ell_\psi\})\,
	T_p(J_\psi,\{\ell_\psi\})^{\frac{p}{2p-1}}
	\\
	&\le
	\lambda_{1,p}(\mathcal G,\mathcal V^D)\,
	T_p(J_\psi,\{\ell_\psi\})^{\frac{p}{2p-1}}
	\\
	&\le
	\lambda_{1,p}(\mathcal G,\mathcal V^D)\,
	T_p(\mathcal G,\mathcal V^D)^{\frac{p}{2p-1}}
	\\
	&=
	C_p^{\mathrm{KJ}}.
\end{align*}
	must in fact be an equality. In particular,
	\(
	T_p(J_\psi,\{\ell_\psi\})
	=
	T_p(\mathcal G,\mathcal V^D),
	\)
	and equality holds in \eqref{eq:KJrear2}. By the layer-cake argument in the
	proof of Lemma~\ref{lem:KJrearrange}, this implies that
	$\kappa\equiv\varphi$. Comparing the derivatives in the defining relation
	for $\kappa$, we obtain
	\[
	\alpha(\varphi(\tau))=R(\tau)
	\]
	for almost every $\tau\in(0,I(0))$. Passing to the limit as
	$\tau\to I(0)$ and using the fact that
	$\alpha(t)\to|\mathcal G|=:L$ as $t\to0$, we deduce that
	\(
	\ell_\psi=R(I(0))=L.
	\)
	Consequently,
	\[
	T_p(\mathcal G,\mathcal V^D)
	=
	T_p(J_\psi,\{\ell_\psi\})
	=
	\left(\frac{L^{q+1}}{q+1}\right)^{p-1}.
	\]
	Hence equality also holds in the Saint-Venant inequality \eqref{eq:SV}, and
	the conclusion follows from Theorem~\ref{thm:SV}.
	\end{proof}


\begin{thebibliography}{99}
	
	\bibitem{AST17}
	R.~Adami, E.~Serra, and P.~Tilli,
	\textit{Nonlinear dynamics on branched structures and networks},
	Rend.\ Semin.\ Mat.\ Univ.\ Politec.\ Torino \textbf{75} (2017),
	no.~1--2, 65--115.
	
	\bibitem{BKKM17}
	G.~Berkolaiko, J.~B.~Kennedy, P.~Kurasov, and D.~Mugnolo,
	\textit{Edge connectivity and the spectral gap of combinatorial and
		quantum graphs},
	J.\ Phys.\ A \textbf{50} (2017), no.~36, 365201.
	
	\bibitem{BK}
	G.~Berkolaiko and P.~Kuchment,
	\textit{Introduction to Quantum Graphs},
	Math.\ Surveys Monogr., vol.~186, Amer.\ Math.\ Soc., Providence, RI,
	2013.
	
	\bibitem{BM25}
	P.~Bifulco and D.~Mugnolo,
	\textit{On the $p$-torsional rigidity of combinatorial graphs},
	Nonlinear Anal.\ \textbf{251} (2025), article no.~113694.
	
	\bibitem{Brasco}
	L.~Brasco,
	\textit{On torsional rigidity and principal frequencies: an invitation
		to the Kohler--Jobin rearrangement technique},
	ESAIM Control Optim.\ Calc.\ Var.\ \textbf{20} (2014), no.~2, 315--338.
	
	\bibitem{DR}
	L.~M.~Del Pezzo and J.~D.~Rossi,
	\textit{The first eigenvalue of the $p$-Laplacian on quantum graphs},
	Anal.\ Math.\ Phys.\ \textbf{6} (2016), no.~4, 365--391.
	
	\bibitem{FGL13}
	I.~Fragal\`a, F.~Gazzola, and J.~Lamboley,
	\textit{Sharp bounds for the $p$-torsion of convex planar domains},
	in: Geometric Properties for Parabolic and Elliptic PDEs,
	Springer INdAM Ser., vol.~2, Springer, Milan, 2013, pp.~97--115.
	
	\bibitem{KJ}
	M.-T.~Kohler-Jobin,
	\textit{Une m\'ethode de comparaison isop\'erim\'etrique de
		fonctionnelles de domaines de la physique math\'ematique},
	Z.\ Angew.\ Math.\ Phys.\ \textbf{29} (1978), no.~5, 757--766.
	
	\bibitem{kuch1}
	P.~Kuchment,
	\textit{Quantum graphs: I. Some basic structures},
	Waves Random Media \textbf{14} (2004), no.~1, S107--S128.
	
	\bibitem{kuch2}
	P.~Kuchment,
	\textit{Quantum graphs: an introduction and a brief survey},
	in: Analysis on Graphs and Its Applications,
	Proc.\ Sympos.\ Pure Math., vol.~77, Amer.\ Math.\ Soc.,
	Providence, RI, 2008, pp.~291--314.
	
	\bibitem{LE11}
	J.~Lang and D.~E.~Edmunds,
	\textit{Eigenvalues, Embeddings and Generalised Trigonometric
		Functions},
	Lecture Notes in Math., vol.~2016, Springer, Heidelberg, 2011.
	
	\bibitem{Lindqvist}
	P.~Lindqvist,
	\textit{Notes on the Stationary $p$-Laplace Equation},
	SpringerBriefs Math., Springer, Cham, 2019.
	
	\bibitem{MuPl}
	D.~Mugnolo and M.~Pl\"umer,
	\textit{On torsional rigidity and ground-state energy of compact
		quantum graphs},
	Calc.\ Var.\ Partial Differ.\ Equ.\ \textbf{62} (2023), no.~1,
	Paper no.~27, 37~pp.
	
	\bibitem{OT24}
	S.~\"Ozcan and M.~T\"aufer,
	\textit{Torsional rigidity on metric graphs with $\delta$-vertex
		conditions},
	J.\ Math.\ Phys.\ \textbf{67} (2026), no.~6, 061508,
	doi:10.1063/5.0301900.
	
	\bibitem{Polya}
	G.~P\'olya,
	\textit{Torsional rigidity, principal frequency, electrostatic
		capacity and symmetrization},
	Quart.\ Appl.\ Math.\ \textbf{6} (1948), 267--277.
	
	\bibitem{PS}
	G.~P\'olya and G.~Szeg\H{o},
	\textit{Isoperimetric Inequalities in Mathematical Physics},
	Ann.\ of Math.\ Stud., vol.~27, Princeton Univ.\ Press, Princeton,
	NJ, 1951.
	
	\bibitem{pol3}
	G.~P\'olya and A.~Weinstein,
	\textit{On the torsional rigidity of multiply connected
		cross-sections},
	Ann.\ of Math.\ (2) \textbf{52} (1950), no.~1, 154--163.
	
	\bibitem{SV}
	B.~de Saint-Venant,
	\textit{M\'emoire sur la torsion des prismes},
	M\'em.\ Savants \'Etrangers \textbf{14} (1855), 233--560.
	
	\bibitem{Vazquez}
	J.~L.~V\'azquez,
	\textit{A strong maximum principle for some quasilinear elliptic
		equations},
	Appl.\ Math.\ Optim.\ \textbf{12} (1984), no.~3, 191--202.
	
\end{thebibliography}
\end{document}